\documentclass[a4paper, 12point]{article}
\usepackage{}
\newcommand{\be}{\begin{equation}}
\newcommand{\ee}{\end{equation}}
\newcommand{\bea}{\begin{eqnarray}}
\newcommand{\eea}{\end{eqnarray}}

\usepackage{graphicx,epsfig}
\usepackage{enumerate}
\usepackage{amssymb}
\usepackage{mathrsfs}
\usepackage{amsthm}
\usepackage{amsmath}
\usepackage{setspace}
\usepackage{endnotes}
\newtheorem{Theorem}{Theorem}[section]
\theoremstyle{definition}
\newtheorem{Definition}{Definition}[section]
\newtheorem{note}{Note}[section]
\theoremstyle{lemma}
\newtheorem{Lemma}{Lemma}[section]
\theoremstyle{example}
\newtheorem{example}{Example}[section]
\newtheorem{Remark}{Remark}[section]
\theoremstyle{illustration}

\theoremstyle{proposition}
\newtheorem{pro}{Proposition}[section]
\theoremstyle{corollary}
\newtheorem{Corollary}{Corollary}[section]
\numberwithin{equation}{section}
\begin{document}
\date{}
\title{\textbf{On some geometric properties of generalized Musielak-Orlicz sequence space and corresponding operator ideals}}
\author{Amit Maji \footnote{Corresponding author's e-mail: amit.iitm07@gmail.com }~ and
        P. D. Srivastava \footnote{ Author's e-mail: pds@maths.iitkgp.ernet.in}
\\
\textit{\small{Department of Mathematics, Indian Institute of Technology Kharagpur}}
\\
\textit{\small{Kharagpur 721 302, West Bengal, India}}}
\maketitle
\begin{center}\textbf{Abstract}\end{center}
Let $\bold{\Phi} =(\phi_n)$ be a Musielak-Orlicz function, $X$ be a real Banach space and $A$ be any infinite matrix.
In this paper, a generalized vector-valued Musielak-Orlicz sequence space $l_{\bold{\Phi}}^{A}(X)$ is introduced.
It is shown that the space is complete normed linear space under certain conditions on the matrix $A$. It is also shown that $l_{\bold{\Phi}}^{A}(X)$ is a $\sigma$- Dedikind complete whenever $X$ is so. We have discussed some geometric properties, namely, uniformly monotone, uniform Opial property for this space. Using the sequence of $s$-number (in the sense of Pietsch), the operators of $s$-type $l_{\bold{\Phi}}^{A}$ and operator ideals under certain conditions on the matrix $A$ are discussed.
\section{Introduction}

\qquad The theory of sequence spaces has several important applications in many branches of mathematical analysis. The classical sequence space $l_2$
is extended to $l_p, 1< p< \infty$ by Reisz \cite{RIE}, and further its generalizations to Lorentz sequence space $l_{p, q}$, for $0< p, q< \infty$ due to Hardy and Littlewood \cite{HARD}. W. Orlicz \cite{ORL} has generalized the sequence space $l_p$ to Orlicz sequence space $l_{\phi}$ with the help of Orlicz function $\phi$ while Woo \cite{WOO} has generalized the Orlicz sequence space to modular sequence space. In recent years, many mathematicians are interested to study the theory of sequence spaces generated by Ces\`{a}ro mean, Orlicz function, Musielak-Orlicz function or using the combination of these. The Ces\`{a}ro-Orlicz sequence space is introduced by Lim and Lee \cite{LIM}. Later on, Cui et al. \cite{CUI} have studied and discussed its basic topological properties as well as geometric properties. In 1990, Kaminska introduced Orlicz-Lorentz space \cite{KAM1} and after that Foralewski et al. (2008, \cite{FOR4}) have introduced generalized  Orlicz-Lorentz sequence space. Srivastava and Ghosh \cite{SRI} also studied vector valued sequence spaces using Orlicz function.

With the help of $s$-numbers, many authors such as Pietsch \cite{PIE6}, Carl \cite{CAR} etc. have introduced and studied operator ideal over the sequence spaces.
The theory of operator ideal is important in spectral theory, the geometry of Banach spaces, theory of eigenvalue distributions, etc. The
$s$-numbers, in particular, approximation numbers of a bounded linear operator plays an important role in the study of compactness, eigenvalue problem etc. Pietsch \cite{PIE1} studied the class of $l_p$, $0<p< \infty$ type operators. In 2012, Gupta and Acharya \cite{GUP2} have introduced the class of $l_{\phi}$ type. Recently, Gupta and Bhar \cite{GUP1} studied generalized Orlicz-Lorentz sequence space and developed operator ideals with the help of $s$-number. Maji and Srivastava \cite{AMIT} also introduced and studied the class of $s$-type $|A, p|$ operators using $s$-number and $|A, p|$ space.

The natural question arises that can we unify the results on various type of operators. The present work is an attempt in this direction. We have introduced a vector-valued new sequence space with the help of sequence of Orlicz functions and an infinite matrix. It is shown that sequence spaces such as Musielak-Orlicz, Ces\`{a}ro-Orlicz, Orlicz-Lorentz etc. are the particular cases with the suitable choice of an infinite matrix. We have proved that the space is a complete normed linear space under certain conditions on the matrix. We have also discussed some geometric properties, namely, uniformly monotone, uniform Opial property for this space. With the help of $s$-number, we have also developed operator ideals under certain conditions on the matrix.

\section{Preliminaries}
Throughout the paper, we consider $(X, \|. \|)$ as a real Banach space and $S(X)$ be the unit sphere in $X$. Let $l_{\infty}(X) = \displaystyle \{ \bar{x}=(x_n): x_n \in X,  ~\sup_{n \geq 1}\|x_n \| < \infty  \}$. Let $w$, $\mathbb{R}$, $\mathbb{R}_{+}$ and $\mathbb{N}$ stand for the set of all real sequences, the set of real numbers, the set of all non negative real numbers and the set of all natural numbers respectively.
A sequence space $(\lambda, \|. \|_{\lambda})$ is called a $BK$ space if it is a Banach space with continuous coordinates $p_n: \lambda \rightarrow \mathbb{R}$, i.e., $p_n({\alpha})= \alpha_n$ for all ${\alpha} \in \lambda$ and every $n\in \mathbb{N}$, where ${\alpha}= (\alpha_1, \alpha_2, \ldots  )$.
A BK-space $(\lambda, \|. \|_{\lambda})$ is called an AK-space if ${\alpha}^{[n]} \rightarrow {\alpha}$, where ${\alpha}^{[n]}= (\alpha_1, \alpha_2, \ldots, \alpha_n, 0, 0, \ldots )$, the nth section of ${\alpha}$. An infinite matrix $A=(a_{nk})$ is called a triangle if $a_{nn}\neq 0$ and $a_{nk}=0$ for all $k>n$.
The following notations are used:\\
$x|_i$ stand for $(x_1, x_2, \ldots, x_i, 0, 0, \ldots)$ and $x|_{\mathbb{N}-i}$ stand for
$(0, 0,  \ldots, 0, x_{(i+1)}, x_{(i+2)}, \ldots)$, where $x \in w$ and $i \in \mathbb{N}$.


\begin{Definition}{\rm \cite{LIN}}
    An Orlicz function $\phi: [0, \infty)\rightarrow [0, \infty)$ is a convex, nondecreasing, continuous function on $[0, \infty)$ such that $\phi(0) = 0$, $\phi(t) > 0$ for $t > 0$ and $\phi(t) \rightarrow \infty$ as $t \rightarrow \infty$.
\end{Definition}
\begin{Definition}{\rm \cite{LIN}}
        An Orlicz function $\phi$ is said to satisfy $\Delta_{2}$-condition at zero ($\phi \in \Delta_2(0)$ for short) if there exist $K > 0$ and $t_{0}> 0$ such that $\phi(2t) \leq K \phi(t)$ for all $t \in [0, t_{0}]$.
\end{Definition}

\begin{Definition}{\rm \cite{LIN}}
A sequence $\bold{\Phi} = (\phi_n)$, where each $\phi_n$ is an Orlicz function, is said to be a Musielak-Orlicz function.
\end{Definition}
The Musielak-Orlicz sequence space $l_{\bold{\Phi}}$ ( see \cite{KAT}, \cite{LIN}) generated
by $\bold{\Phi}= (\phi_n)$ is defined as $$ l_{\bold{\Phi}} = \Big \{ x=(x_{n}) \in w :  \sum_{n=1}^{\infty}\phi_n \Big(\frac{|x_{n}|}{ \sigma}\Big) < \infty \quad {\rm{ for ~ some}}~~ \sigma > 0 \Big \}$$ and convex modular $\varrho_{\bold{\Phi}}$ is defined as $\varrho_{\bold{\Phi}}(x) = \displaystyle\sum_{n=1}^{\infty}\phi_n (|x_{n}|)$.

A Musielak-Orlicz function $\bold{\Phi} = (\phi_n)$ is said to satisfy the condition $\delta_2$ ($\bold{\Phi} \in \delta_2 $ for short) if there exists $K > 0$, $\delta> 0$ and a nonnegative scalar sequence $(c_n) \in  l_1$ such that for each $n \in \mathbb{N}$ and all $x \geq 0$,
$$ \phi_n(2x) \leq K \phi_{n}(x) + c_n $$ whenever $\phi_{n}(x) \leq \delta $.

It is easy to check that $\bold{\Phi} = (\phi_n) \in \delta_2$ if and only if for all $\beta > 1$ there exists $K >0$, $ \delta > 0$
and a nonnegative scalar sequence $(c_n) \in l_1$ such that for all $n \in \mathbb{N}$ and $x \geq 0$
$$ \phi_n(\beta x) \leq K \phi_{n}(x) + c_n $$ holds if $\phi_n(x) \leq \delta$. For details on Musielak-Orlicz function, one can see \cite{KAT}, \cite{LIN}.

A Musielak-Orlicz function $\bold{\Phi} = (\phi_n)$ is said to satisfy the condition $(*)$ (see \cite{KAM}) if for any $\epsilon\in (0, 1)$ there exists $\delta>0$ such that
\begin{align}{\label{2.1}}
\phi_n(u) < 1 - \epsilon \mbox{~implies~} \phi_n((1 + \delta)u) \leq 1\mbox{~for all~} n \in \mathbb{N} \mbox{~and~} u\geq 0.
\end{align}

\begin{Definition}{{\rm\cite{KAN}}}
A real Banach space $X$, endowed with partial order $\leq$ satisfying the properties $x \leq y$ implies $x+ z \leq y + z$ for all $x, y, z \in X$ and
$0 \leq tx$ for $0 \leq x$ and $t \in \mathbb{R}_{+}$, is said to be a Banach lattice if $\|x \| \leq \| y\|$ whenever $|x| < |y|$ for $x, y \in X$, where $|x| = \displaystyle \sup \{x, -x\}$, the absolute value of $x$.
\end{Definition}

Let $X=(X, \leq , \|. \|)$ be a Banach lattice with a lattice norm $\|.\|$.
The positive cone $X_{+}$ of $X$ is defined as $ X_{+} = \{ x \in X : 0 \leq x \}$.
A Banach lattice $X$ is said to be $\sigma$-Dedekind complete ($\sigma$-DC for short) if any nonnegative order bounded
sequence $(x_n)$ in $X$ has supremum in $X$.

The norm $\|. \|$ in a Banach lattice $X$ is said to be strictly monotone if $x, y \in X_{+}$ with $y \leq x$ and $y \neq x$, there holds
$\|y \| < \|x\|$ and the space $X$ has strictly monotone property. The norm $\|. \|$ is said to be uniformly monotone (UM for short)
if for each $\epsilon> 0$ there exists $\delta(\epsilon)>0$ such that $$\|x + y \| \geq 1 + \delta(\epsilon),$$
for each $x, y \in X_{+}$ with $\|x \| =1$ and $\|y \| \geq \epsilon$. Then the space $X$ has UM property. A Banach lattice $X$ is said to be an
AL-space if $\|x+y \| = \|x \| + \| y \|$ for all $x, y \in X_{+}$. An element $ x\in X$ is said to be order continuous if for any sequence $(x_n)$
in $X$ such that $0 \leq x_n \leq |x|$ for all $n \in \mathbb{N}$, $x_1 \geq x_2 \geq \ldots$ and $\inf x_n =0$, there holds $\|x_n \| \rightarrow 0$.
The space $X$ is said to be order continuous if every element of $X$ is order continuous. For details on Banach lattice, refer to \cite{HUD}, \cite{KAN}.\\

A Banach space $X$ is said to have Opial property if for every weakly null sequence $(x_n)$ and for every $x \neq 0$ in $X$, we have
$$\displaystyle\liminf_{n \rightarrow \infty} \|x_n \| < \liminf_{n \rightarrow \infty} \|x_n + x \|.$$
A Banach space $X$ is said to have the uniform Opial property if for every $\epsilon > 0$ there exists $\mu > 0$ such that
$$ 1 + \mu \leq \liminf_{n \rightarrow \infty}\|x_n + x\|$$ for any weakly null sequence $(x_n) \in S(X)$ and $x \in X$ with $\|x \| \geq \epsilon$.\\


Let $ \mathcal{L} $ be the class of all bounded linear operators between two arbitrary Banach spaces and $\mathcal{L}(E, F)$ be the space of
all bounded linear operators from a Banach space $E$ to a Banach space $F$. We denote $E^{'}$ as the dual of $E$ and $x^{'}$ as the continuous
linear functional on $E$. Define $x'\otimes y: E \rightarrow F$ by $x'\otimes y(x) = x'(x)y$ for all $x \in E$ and $y \in F$.
\begin{Definition}\textbf{($s$-numbers of a bounded linear operator)}{\rm{(\cite{CAR}, \cite{PIE6})}}\label{def1}
        A map $ s=(s_{n}): \mathcal{L} \rightarrow \mathbb{R^{\mathbb{N}}}$ assigning to every operator
        $T \in \mathcal{L} $ a non-negative scalar sequence $ (s_{n}(T))$ is called an $s$-number sequence if the following conditions are satisfied:
\begin{description}
        \item [($S1$)] monotonicity:  $ \| T \| =  s_{1}(T) \geq s_{2}(T)\geq \cdots \geq 0$, \quad for $T \in \mathcal{L}(E, F) $
        \item [($S2$)] additivity: $s_{m+n-1}(S+T) \leq s_{m}(S) + s_{n}(T),$ \quad for
                        $ S, T \in \mathcal{L}(E, F)$, $m, n \in \mathbb{N}$
        \item [($S3$)] ideal property: $s_{n}(RST) \leq \| R \|s_{n}(S)\| T \|,$ \quad for some
                        $R \in \mathcal{L}(F, F_{0}),~ S \in \mathcal{L}(E, F)$ and $ T \in \mathcal{L}(E_{0}, E)$, where $ E_{0}, F_{0} $ are arbitrary Banach spaces
        \item [($S4$)] rank property: if $\mbox{rank}(T) \leq n ~{\rm then}~ s_{n}(T) = 0$
        \item [($S5$)] norming property: $ s_{n}(I : l_{2}^n \rightarrow l_{2}^n) = 1$, where $I$
                       denotes the identity operator on the $n$-dimensional Hilbert space $l_{2}^{n}$.
\end{description}
\end{Definition}
        We call $s_{n}(T)$ the $n$-th $s$-number of the operator $T$. In particular the approximation number is an example of an $s$-number and
        the $n$-th approximation number, denoted by $a_{n}(T)$, is defined as
\begin{center}
        $ a_{n}(T)  = \inf \Big\{  \| T - L \| : \quad L \in \mathcal{L}(E, F), ~ \rm{rank}(L) < n  \Big\}$.
\end{center}
\begin{Definition}{\rm(\cite{PIE6}, p.90)}
        An $s$-number sequence $s=(s_{n})$ is called injective if, given any metric injection $ J \in \mathcal{L}(F, F_{0})$,
         $ s_{n}(T) = s_{n}(JT)$ for all $T \in \mathcal{L}(E, F)$.
\end{Definition}
\begin{Definition}{\rm(\cite{PIE6}, p.95)}
        An $s$-number sequence $s=(s_{n})$ is called surjective if, given any metric surjection
        $ Q \in \mathcal{L}(E_{0}, E)$, $ s_{n}(T) = s_{n}(TQ)$ for all $T \in \mathcal{L}(E, F)$.
\end{Definition}
\begin{Definition}\textbf{(Operator ideals) } (\cite{PIE4},~ \cite{RET})\label{def2}
		A sub collection $\mathcal{M}$ of $\mathcal{L}$ is said to be an operator ideal if each component
        $\mathcal{M}(E,F) = \mathcal{M} \bigcap \mathcal{L}(E,F)$ satisfies the following conditions:
\begin{itemize}
\item [($OI1$)]if $x' \in E'$, $y \in F$ then $x'\otimes y \in \mathcal{M}(E, F)$
\item [($OI2$)]if $S, T \in \mathcal{M}(E, F)$ then $S+T \in \mathcal{M}(E, F)$
\item [($OI3$)]if $S \in \mathcal{M}(E, F)$, $ T \in \mathcal{L}(E_{0}, E)$ and
        $R \in \mathcal{L}(F, F_{0})$ then $RST \in \mathcal{M}(E_{0}, F_{0})$.
\end{itemize}
\end{Definition}
		
\begin{Definition}(\cite{PIE4},~ \cite{RET})
        A function $ \alpha : \mathcal{M} \rightarrow \mathbb{R}^{+}  $ is said to be a quasi-norm
        on the ideal $ \mathcal{M} $ if the following conditions hold:
\begin{itemize}
\item[($QN1$)] if $x' \in E' $, $y \in F$ then $ \alpha(x'\otimes y)= \| x^{'} \| \| y \| $
\item[($QN2$)] if $S, T \in \mathcal{M}(E, F)$ then there exists a constant $ C \geq 1$
                such that $ \alpha(S+T)\leq C [\alpha(S) + \alpha(T)]$
\item[($QN3$)] if $S \in \mathcal{M}(E, F)$, $ T \in \mathcal{L}(E_{0}, E)$ and
                $R \in \mathcal{L}(F, F_{0})$, then $ \alpha(RST) \leq \|R \|\alpha(S ) \| T \|$.
\end{itemize}
        In particular if $ C =1 $ then $\alpha$ becomes a norm on the operator ideal
        $ \mathcal{M} $.
\end{Definition}

        An ideal $ \mathcal{M} $ with a quasi-norm $ \alpha $, denoted by $ [ \mathcal{M}, \alpha] $
        is said to be a quasi-Banach operator ideal if each component $ \mathcal{M}(E, F)$ is
        complete under the quasi-norm $ \alpha $.

\section{Main results}
        Let $A= (a_{nk})$ be an infinite real matrix, $\bold{\Phi} =(\phi_n)$ be a Musielak-Orlicz function and $X$ be a real Banach space.
        Now we introduce generalized Musielak-Orlicz sequence space using an infinite matrix $A =(a_{nk})$ as $$ \displaystyle l^{A}_{\bold{\Phi}}(X) = \Big \{ \bar{x}=(x_{k}) \in l_{\infty}(X) :  \sum_{n=1}^{\infty} \phi_n \bigg(\frac{\sum\limits_{k=1}^{\infty}\|a_{nk}x_{k} \|}{ \sigma}\bigg) < \infty \quad {\rm{ for ~ some}}~~ \sigma > 0 \Big \}$$
        and
        $$ \displaystyle h^{A}_{\bold{\Phi}}(X) = \Big \{ \bar{x}=(x_{k}) \in l_{\infty}(X) :  \sum_{n=1}^{\infty} \phi_n \bigg(\frac{\sum\limits_{k=1}^{\infty}\|a_{nk}x_{k} \|}{ \sigma}\bigg) < \infty \quad {\mbox{ for all}}~~ \sigma > 0 \Big \}.$$

        It can be proved that the space $ \displaystyle l^{A}_{\bold{\Phi}}(X)$ is a linear space and the space $ \displaystyle h^{A}_{\bold{\Phi}}(X)$
        is a linear subspace of $ \displaystyle l^{A}_{\bold{\Phi}}(X)$. If $X = \mathbb{R}$ then we write $ \displaystyle l^{A}_{\bold{\Phi}}$ and $ \displaystyle h^{A}_{\bold{\Phi}}$ instead of $ \displaystyle l^{A}_{\bold{\Phi}}(\mathbb{R})$ and $ \displaystyle h^{A}_{\bold{\Phi}}(\mathbb{R})$ respectively. Let us define a function $\varrho_{\bold{\Phi}}^{A}$ on $\displaystyle l^{A}_{\bold{\Phi}}(X)$ by
        $$ \varrho_{\bold{\Phi}}^{A}(\bar x) =  \sum_{n=1}^{\infty} \phi_n \bigg(\sum\limits_{k=1}^{\infty}\|a_{nk}x_{k} \| \bigg).$$
        Let $\mathcal{A}$ be a class of infinite matrices $A= (a_{nk})$ such that each column of a matrix is non zero, i.e., for each $k$ there exists at least one $n_0$ such that $a_{n_{0}k} \neq 0$. If $A=(a_{nk}) \in \mathcal{A}$ then it can be proved that the function
        $\varrho_{\bold{\Phi}}^{A}$ is a convex modular on $\displaystyle l^{A}_{\bold{\Phi}}(X)$.

\begin{pro}
The following conditions are equivalent.\\
(1). $ \displaystyle l^{A}_{\bold{\Phi}}(X) \neq \{\bar{0} \}.$\\
(2). For some $k \in \mathbb{N}$, $(|a_{nk}|)_{n=1}^{\infty} \in l_{\bold{\Phi}}.$
\end{pro}

\begin{proof}
$(1) \Rightarrow (2)$.\\
Let $\bar{0} \neq \bar{x} \in \displaystyle l^{A}_{\bold{\Phi}}(X)$. Then there exists some $l \in \mathbb{N}$ such that
$x_l \in X$ and $x_l \neq 0$. Clearly $\bar{y} = (0, 0, \ldots, 0, x_l, 0, 0, \ldots) \in \displaystyle l^{A}_{\bold{\Phi}}(X).$
Thus for some $t> 0$, we have $\varrho_{\bold{\Phi}}^{A}(t \bar{y}) < \infty$. Now
\begin{align*}
\varrho_{\bold{\Phi}}^{A}(t \bar{y}) & = \sum_{n=1}^{\infty} \phi_n \bigg(t \sum\limits_{k=1}^{\infty}\|a_{nk}y_{k} \| \bigg)
= \sum_{n=1}^{\infty} \phi_n \Big (t |a_{nl}| \|x_{l}\| \Big).
\end{align*}
Choose $t_0 = t \|x_l \|>0$. Then
$\varrho_{\bold{\Phi}}^{A}(t \bar{y}) =\displaystyle \sum_{n=1}^{\infty} \phi_n \Big (t_0 |a_{nl}| \Big)= \varrho_{\bold{\Phi}}\Big(t_0 (|a_{nl}|)_{n=1}^{\infty}\Big).$
Hence $(|a_{nl}|)_{n=1}^{\infty} \in l_{\bold{\Phi}}$.\\\\
$(2) \Rightarrow (1)$.\\
Let for some $k \in \mathbb{N}$, $(|a_{nk}|)_{n=1}^{\infty} \in l_{\bold{\Phi}}.$ Define $\bar{x} = (0, 0, \ldots, 0, x_l, 0, 0, \ldots)$. Then $ \bar{x} \in \displaystyle l_{\infty}(X).$ Now for $k=l$, $t = \|x_l \|$, $\varrho_{\bold{\Phi}} \Big(t (|a_{nl}|)_{n=1}^{\infty}\Big)= \displaystyle \sum_{n=1}^{\infty} \phi_n \Big ( \|x_{l}\| |a_{nl}| \Big) < \infty$. Again
\begin{align*}
\varrho_{\bold{\Phi}}^{A}(1.\bar{x})  = \sum_{n=1}^{\infty} \phi_n \bigg(\sum\limits_{k=1}^{\infty}\|a_{nk}x_{k} \|\bigg)
= \sum_{n=1}^{\infty} \phi_n \Big(\|x_{l} \| |a_{nl}| \Big )
= \varrho_{\bold{\Phi}}\Big(t (|a_{nl}|)_{n=1}^{\infty}\Big) < \infty.
\end{align*}
This implies $\bar{x} \in \displaystyle l^{A}_{\bold{\Phi}}(X)$. This completes the proof.
\end{proof}
\textbf{Applications:}\\
The linear space $l_{\bold{\Phi}}^{A}(X)$ contains many known sequence spaces as particular cases with the suitable choice of the matrix $A =(a_{nk})$. For example:\\
1. Choose the matrix $A=(a_{nk})$ as an identity matrix and $X= \mathbb{R}$, then $l_{\bold{\Phi}}^{A}(X)$ reduces to Musielak-Orlicz sequence space $\cite{KAT}$. In addition, if for all $n \in \mathbb{N}$, $\phi_n = \phi$, an Orlicz function then the space gives Orlicz sequence space.\\
2. Choose the matrix $A=(a_{nk})$ as a Ces\`{a}ro matrix of order $1$ and $X= \mathbb{R}$, then $l_{\bold{\Phi}}^{A}(X)$ reduces to Ces\`{a}ro-Orlicz sequence space $\cite{KUB}$.\\
3. Choose the matrix $A=(a_{nk})$ such that $a_{nk} = n^{\frac{1}{p} - \frac{1}{q}}$ for $n =k$ and $0$ otherwise, where $p, q >0$ and suppose the non-increasing
rearrangement of $\bar{x} = (x_n) \in l_{\infty}(X)$, denoted by $(x_n^{*})$ is given as
$$x_n^{*} = \inf \{c \geq 0: | \{i \in \mathbb{N}: \|x_i \| > c \}| < n \},$$
where the vertical bars indicate number of elements in the enclosed set,
  then $l_{\bold{\Phi}}^{A}(X)$ reduces to Orlicz-Lorentz space \cite{FOR3}. In addition, if $\phi(x) = x^q$, then the space reduces to Lorentz space $l_{p, q}$ \cite{PIE6}.\\
4. Choose $\phi_n = x^p, 0<p< \infty$ for all $n$ and $X= \mathbb{R}$, then $l_{\bold{\Phi}}^{A}(X)$ reduces to $A-p$ spaces denoted by $|A, p|$ studied by Rhoades
\cite{RHO1}.

Throughout the paper, we assume that $A=(a_{nk}) \in \mathcal{A}$. Let $ \bar{x} \in l_{\bold{\Phi}}^{A}(X)$ and define
\begin{equation}{\label{eq1}}
\|\bar{x} \|_{\bold{\Phi}}^{A} = \inf \Big \{ \sigma > 0 : \sum_{n=1}^{\infty} \phi_n \bigg(\frac{\sum\limits_{k=1}^{\infty}\|a_{nk}x_{k} \|}{ \sigma}\bigg) \leq 1 \Big \}.
\end{equation}
\begin{Theorem}
The space $ \displaystyle l^{A}_{\bold{\Phi}}(X)$ is a complete normed linear space with the norm $ \|. \|_{\bold{\Phi}}^{A}$ as defined by $(\ref{eq1}).$
\end{Theorem}
\begin{proof}
First we prove that $ \|. \|_{\bold{\Phi}}^{A}$ is a norm on the space  $ \displaystyle l^{A}_{\bold{\Phi}}(X)$.\\
Let $\|\bar{x} \|_{\bold{\Phi}}^{A} = 0$. Then by definition given by \eqref{eq1}, for all $\sigma > 0$ there exists $C>0$ such that for all $n$
\begin{align*}
\frac{\sum\limits_{k=1}^{\infty}\|a_{nk}x_{k} \|}{ \sigma}  \leq C
\quad \Rightarrow |a_{nk}| \|x_{k} \|  \leq C \sigma
\end{align*}
holds for all $n$ and any arbitrary $\sigma >0$. Thus $\|x_{k} \| =0$, hence $x_k =0$ for all $k$. Therefore $\bar{x} = \bar 0$.
Clearly if $\bar{x} = \bar 0$, then $\|\bar{x} \|_{\bold{\Phi}}^{A} = 0$.\\
Let $\bar{x} =(x_k), \bar{y} =(y_k) \in l^{A}_{\bold{\Phi}}(X)$. Let $\epsilon > 0$ be any number. Then there exists $\sigma_1> 0, \sigma_2> 0$ such that
$$ \sum_{n=1}^{\infty} \phi_n \bigg(\frac{\sum\limits_{k=1}^{\infty}\|a_{nk}x_{k} \|}{ \sigma_1}\bigg) \leq 1, \sigma_1 \leq \|\bar{x} \|_{\bold{\Phi}}^{A} + \frac{\epsilon}{2} \mbox{~and~}
\sum_{n=1}^{\infty} \phi_n \bigg(\frac{\sum\limits_{k=1}^{\infty}\|a_{nk}y_{k} \|}{ \sigma_2}\bigg) \leq 1, \sigma_2 \leq \|\bar{y} \|_{\bold{\Phi}}^{A} + \frac{\epsilon}{2}.$$
Using the convexity property of each $\phi_n$, we get
\begin{align*}
\sum_{n=1}^{\infty} \phi_n \bigg(\frac{\sum\limits_{k=1}^{\infty}\|a_{nk}(x_k + y_{k}) \|}{\sigma_1 + \sigma_2}\bigg) & \leq  \frac{\sigma_1}{\sigma_1 + \sigma_2} \sum_{n=1}^{\infty} \phi_n \bigg(\frac{\sum\limits_{k=1}^{\infty}\|a_{nk}x_k \|}{\sigma_1}\bigg) + \frac{\sigma_2}{\sigma_1 + \sigma_2}\sum_{n=1}^{\infty} \phi_n \bigg(\frac{\sum\limits_{k=1}^{\infty}\|a_{nk} y_{k}\|}{\sigma_2}\bigg) \\
& \leq 1.
\end{align*}
Thus $$\|\bar{x} + \bar{y}\|_{\bold{\Phi}}^{A} \leq {\sigma_1 + \sigma_2} \leq \|\bar{x} \|_{\bold{\Phi}}^{A} + \| \bar{y}\|_{\bold{\Phi}}^{A} + \epsilon.$$
Since $\epsilon > 0$ is arbitrary, we have $\|\bar{x} + \bar{y}\|_{\bold{\Phi}}^{A}  \leq \|\bar{x} \|_{\phi}^{A} + \| \bar{y}\|_{\bold{\Phi}}^{A}$. It is easy to show that
$\| \alpha \bar{x} \|_{\bold{\Phi}}^{A} = |\alpha| \| \bar{x} \|_{\bold{\Phi}}^{A}$ for any scalar $\alpha$. Hence $\| .\|_{\bold{\Phi}}^{A}$ is a norm on the space $l_{\bold{\Phi}}^{A}(X)$.\\
To show $l_{\bold{\Phi}}^{A}(X)$ is complete, let $(\bar{x}^{(m)})$ be a Cauchy sequence in $l_{\bold{\Phi}}^{A}(X)$. Then for each $\epsilon > 0$, there exists $m_0 \in \mathbb{N}$
such that $\|\bar{x}^{(m)} - \bar{x}^{(l)}\|_{\bold{\Phi}}^{A} < \epsilon$ for all $m, l \geq m_0$, i.e.,
\begin{align}{\label{eq2}}
\sum_{n=1}^{\infty} \phi_n \bigg(\frac{\sum\limits_{k=1}^{\infty}\|a_{nk}(x^{(m)}_k - x^{(l)}_{k}) \|}{\epsilon}\bigg) \leq 1.
\end{align}
Thus the sequence $ \Big \{\frac{\|x_{k}^{(m)} - x^{(l)}_{k}\|}{ \epsilon}\Big \}$ is bounded. Therefore there exists $C > 0$ such that
$$\frac{\|x_{k}^{(m)} - x^{(l)}_{k}\|}{ \epsilon} \leq C \mbox{~~for all~} m, l \geq m_0$$
Hence $(x_k^{(m)})$ is a Cauchy sequence in $X$ for each $k$. Since $X$ is a Banach space, so $(x_k^{(m)})$ is convergent in $X$.
Let $x_k = \displaystyle\lim_{m \rightarrow \infty}x_k^{(m)}$ for each $k$.
Since each $\phi_k$ is continuous so taking $m \rightarrow \infty$, we get from (\ref{eq2}),
$$\sum_{n=1}^{\infty} \phi_n \bigg(\frac{\sum\limits_{k=1}^{\infty}\|a_{nk}(x_k - x^{(l)}_{k}) \|}{\epsilon}\bigg) \leq 1$$ for all $l \geq m_0$.
Thus $\|\bar{x} - \bar{x}^{(l)}\|_{\bold{\Phi}}^{A} < \epsilon$ for all $l\geq m_0$. Thus $(\bar{x}^{(l)})$ converges to $\bar{x}$ in $l_{\bold{\Phi}}^{A}(X)$.
Since $\bar{x} - \bar{x}^{(m_0)} \in l_{\bold{\Phi}}^{A}(X)$, so $\bar{x} = \bar{x}^{(m_0)} + (\bar{x} - \bar{x}^{(m_0)}) \in l_{\bold{\Phi}}^{A}(X)$.
This completes the proof.
\end{proof}


\begin{Theorem}{\label{thm3}}
The space $ \displaystyle h^{A}_{\bold{\Phi}}(X)$ is an AK-BK space.
\end{Theorem}
\begin{proof}
Clearly $ \displaystyle h^{A}_{\bold{\Phi}}(X)$ is a linear subspace of $ \displaystyle l^{A}_{\bold{\Phi}}(X)$.
To prove $ \displaystyle h^{A}_{\bold{\Phi}}(X)$ is a BK space, it is enough to prove that the space $ \displaystyle h^{A}_{\bold{\Phi}}(X)$ is a
closed subspace of $ \displaystyle l^{A}_{\bold{\Phi}}(X)$. Let $\bar{x}$ belongs to closure of $ \displaystyle h^{A}_{\bold{\Phi}}(X)$ in the norm topology of $ \displaystyle l^{A}_{\bold{\Phi}}(X)$. Then there exists a sequence $(\bar{y}^{(m)})$ in $ \displaystyle h^{A}_{\bold{\Phi}}(X)$, where
$\bar{y}^{(m)} = (y_k^{(m)})_{k \geq 1}$
such that $\|\bar{y}^{(m)} - \bar{x} \|_{\bold{\Phi}}^{A} \rightarrow 0$ as $m \rightarrow \infty$. So for any $\epsilon > 0$
there exists $m_0 \in \mathbb{N}$ such that
\begin{align*}
\sum_{n=1}^{\infty} \phi_n \bigg(\frac{\sum\limits_{k=1}^{\infty}\|a_{nk}(x_k - y_k^{(m)})\|}{\epsilon}\bigg) \leq 1 \quad \mbox{for all} ~m \geq m_0.
\end{align*}
Now for any $\epsilon >0$,
\begin{align*}
\sum_{n=1}^{\infty} \phi_n \bigg(\frac{\sum\limits_{k=1}^{\infty}\|a_{nk}x_k\|}{2 \epsilon}\bigg) & 
 \leq \sum_{n=1}^{\infty} \phi_n \bigg(\frac{1}{2}\frac{\sum\limits_{k=1}^{\infty}\|a_{nk}(x_k - y_k^{(m_0)})\|}{\epsilon} + \frac{1}{2}\frac{\sum\limits_{k=1}^{\infty}\|a_{nk}y_k^{(m_0)}\|}{\epsilon} \bigg)\\
& \leq \frac{1}{2} \bigg[\sum_{n=1}^{\infty} \phi_n \bigg(\frac{\sum\limits_{k=1}^{\infty}\|a_{nk}(x_k - y_k^{(m_0)})\|}{\epsilon}\bigg)  + \sum_{n=1}^{\infty} \phi_n \bigg(\frac{\sum\limits_{k=1}^{\infty}\|a_{nk}y_k^{(m_0)}\|}{\epsilon} \bigg) \bigg]
 < \infty.
\end{align*}
Thus $\bar{x} \in \displaystyle h^{A}_{\bold{\Phi}}(X)$.

To show $ \displaystyle h^{A}_{\bold{\Phi}}(X)$ is an AK-space, let $\bar x \in \displaystyle h^{A}_{\bold{\Phi}}(X)$.
Now for each $\epsilon, 0< \epsilon<1$ there exists $n_{0} \in \mathbb{N}$ such that
$\displaystyle\sum_{n = n_{0}}^{\infty} \phi_n \bigg(\frac{\sum\limits_{k=1}^{\infty}\|a_{nk}x_k\|}{\epsilon}\bigg) \leq 1.$
Now for $m \geq n_0$,
\begin{align*}
\|\bar x -\bar{x}^{[m]} \|_{\bold{\Phi}}^{A} & = \inf \bigg\{ \sigma > 0: \sum_{n = m+1 }^{\infty} \phi_n \bigg(\frac{\sum\limits_{k=1}^{\infty}\|a_{nk}x_k\|}{\sigma}\bigg) \leq 1 \bigg \}\\
&\leq \inf \bigg\{ \sigma > 0: \sum_{n = m }^{\infty} \phi_n \bigg(\frac{\sum\limits_{k=1}^{\infty}\|a_{nk}x_k\|}{\sigma}\bigg) \leq 1 \bigg \}\leq \epsilon.
\end{align*}
This shows that $x^{[m]} \rightarrow x$ as $m \rightarrow \infty$ under the norm $\|. \|_{\bold{\phi}}^{A}$. Hence $ \displaystyle h^{A}_{\bold{\Phi}}(X)$ is an AK space. This completes the proof.
\end{proof}

\begin{Theorem}{\label{thm40}}
If a Musielak-Orlicz function $\bold{\Phi} = (\phi_{n}) \in \delta_2$, then $ \displaystyle l^{A}_{\bold{\Phi}}(X) = \displaystyle h^{A}_{\bold{\Phi}}(X)$.
\end{Theorem}
\begin{proof}
It suffices to show that $\displaystyle l^{A}_{\bold{\Phi}}(X) \subseteq \displaystyle h^{A}_{\bold{\Phi}}(X)$. Let $\bar{x} \in \displaystyle l^{A}_{\bold{\Phi}}(X)$. Then for some $t>0$, we have
\begin{center}
$\varrho_{\bold{\Phi}}^{A}(t\bar{x}) =\displaystyle \sum_{n=1}^{\infty} \phi_n \bigg(t \sum\limits_{k=1}^{\infty}\|a_{nk}x_k\| \bigg) < \infty.$
\end{center}
Therefore for $\delta_{0}> 0$ there exists $n_0 \in \mathbb{N}$ such that $\displaystyle \phi_n \bigg(t \sum\limits_{k=1}^{\infty}\|a_{nk}x_k\| \bigg)< \delta_{0}$
for all $n \geq n_{0}$. We have to show that for any $r> 0$, $\varrho_{\bold{\Phi}}^{A}(r \bar{x}) < \infty$. If $r$ be any number such that $r \leq t$, then by the non-decreasing property of $\phi_n$ for all $n$, we have
$ \varrho_{\bold{\Phi}}^{A}(r\bar{x}) \leq \varrho_{\bold{\Phi}}^{A}(t\bar{x}) < \infty.$

If $r> t$, choose $\beta>1$ such that $r < \beta t$. Thus
\begin{center}
$ \phi_n \bigg(r \sum\limits_{k=1}^{\infty}\|a_{nk}x_k\| \bigg) \leq  \phi_n \bigg(\beta t \sum\limits_{k=1}^{\infty}\|a_{nk}x_k\| \bigg).$
\end{center}
Since $\bold{\Phi} =(\phi_n) \in \delta_2$, so there exist $K > 0$, $\delta_0 > 0$ and a non negative sequence $(c_n) \in l_1$ such that
for $n \geq n_0$
$$\phi_n \bigg(\beta t \sum\limits_{k=1}^{\infty}\|a_{nk}x_k\| \bigg) \leq K \phi_n \bigg( t \sum\limits_{k=1}^{\infty}\|a_{nk}x_k\| \bigg) + c_n, $$
whenever $\phi_n \bigg(t \sum\limits_{k=1}^{\infty}\|a_{nk}x_k\| \bigg) < \delta_0.$ Therefore
\begin{center}
$ \displaystyle \sum_{n=n_0}^{\infty}\phi_n \bigg(\beta t \sum\limits_{k=1}^{\infty}\|a_{nk}x_k\| \bigg) \leq K \displaystyle \sum_{n=n_0}^{\infty} \phi_n \bigg( t \sum\limits_{k=1}^{\infty}\|a_{nk}x_k\| \bigg) + \displaystyle \sum_{n=n_0}^{\infty} c_n. $
\end{center}
Since $ \varrho_{\bold{\Phi}}^{A}(t\bar{x})< \infty$, so we have $ \displaystyle \sum_{n=n_0}^{\infty}\phi_n \bigg(r \sum\limits_{k=1}^{\infty}\|a_{nk}x_k\| \bigg)< \infty$. Thus for any $r> 0$, we have $\varrho_{\bold{\Phi}}^{A}(r\bar{x}) < \infty.$ Hence $\bar{x} \in \displaystyle h^{A}_{\bold{\Phi}}(X).$ This completes the proof.
\end{proof}

\begin{example}
Let $\phi_n(u) = u^{p_n}\Big( log(1 + u) +1 \Big)$ for $u \geq 0$, $1 \leq p_n < \infty$ with $\displaystyle \sup_{n \geq 1} p_n < \infty$.
Then $\bold{\Phi} = (\phi_{n}) $ is a Musielak-Orlicz function and satisfies $\delta_2$ condition. Hence $ \displaystyle l^{A}_{\bold{\Phi}}(X) = \displaystyle h^{A}_{\bold{\Phi}}(X)$.
In addition if we take $A$ as an identity matrix then $ \displaystyle l_{\bold{\Phi}}(X) \subseteq l(\{p_n \}, X)$, where $ l(\{p_n \}, X)$ denotes vector-valued
Nakano sequence space.
\end{example}

Now we state some known lemmas which will be needed in the sequel. Proofs of the corresponding lemmas run parallel lines as in the references (see \cite{CUI2}, \cite{CUI}, \cite{KAM}).
\begin{Lemma}{\label{2}}
Let $\bar{x} \in \displaystyle h^{A}_{\bold{\Phi}}(X) $ be an arbitrary element. Then
$\|\bar{x}\|_{\bold{\Phi}}^A=1$ if and only if $\varrho_{\bold{\Phi}}^{A}(\bar{x})=1$.
\end{Lemma}

\begin{Lemma}{\label{4}}
Let $\Phi\in\delta_2$. Then for any $\bar{x} \in \displaystyle l^{A}_{\bold{\Phi}}(X)$,
\begin{center}
 $\|\bar x\|_{\bold{\Phi}}^A =1 \mbox{~~if and only if~~} \varrho_{\bold{\Phi}}^{A}(\bar x)=1 $.
\end{center}
\end{Lemma}

\begin{Lemma}{\label{3}}
Let $\bold{\Phi} \in \delta_2$. Then for any sequence $(\bar{x}^{(l)})$ in
$\displaystyle l^{A}_{\bold{\Phi}}(X) $, $\|\bar{x}^{(l)}\|_{\bold{\Phi}}^A \rightarrow 0$ if and only if
$\varrho_{\bold{\Phi}}^{A}(\bar{x}^{(l)}) \rightarrow 0$.
\end{Lemma}

\begin{Lemma}{\label{5}}
Let $\Phi\in \delta_2$ and satisfies the condition
$(*)$ given by \eqref{2.1}. Then for any $\bar{x} \in l_{\bold{\Phi}}^{A}(X)$ and every
$\epsilon\in (0,1)$ there exists $\delta(\epsilon)\in (0, 1)$ such
that $\varrho_{\bold{\Phi}}^{A}(\bar x)\leq 1-\epsilon$ implies $\|\bar x\|_\bold{\Phi}^A \leq
1-\delta$.
\end{Lemma}

\begin{Lemma}{\label{10}}
Let $(X, \|.\|)$ be a normed space. If $f:X\rightarrow\mathbb{R}$ is a convex function in the set $K=\{x\in X: \|x\|\leq1\}$ and
$|f(x)|\leq M$ for all $x\in K$ and some $M>0$ then $f$ is almost uniformly continuous in $K$ i.e., for all $d\in(0,1)$
and $\epsilon>0$ there exists $\delta>0$ such that $\|y\|\leq d$ and $\|x-y\|<\delta$ implies $|f(x)-f(y)|<\epsilon$ for
all $x, y\in K$.
\end{Lemma}

\begin{Lemma}{\label{6}}
Let $\bold{\Phi}\in \delta_2$ and satisfies the
condition $(*)$ given by \eqref{2.1}. Then for each $d\in(0, 1)$ and
$\epsilon>0$ there exists $\delta=\delta(d, \epsilon)>0$ such that
$\varrho_{\bold{\Phi}}^{A}(\bar x)\leq d$, $\varrho_{\bold{\Phi}}^{A}(\bar y)\leq \delta $ imply
\begin{equation}
|\varrho_{\bold{\Phi}}^{A}(\bar x+ \bar y)-\varrho_{\bold{\Phi}}^{A}(\bar x)|<\epsilon \mbox{~for any ~} \bar x, \bar y\in
l_{\bold{\Phi}}^{A}(X).
\end{equation}
\end{Lemma}
\begin{proof}
Let $\epsilon>0$ be any number. Assume that for each $d \in (0, 1)$ and $\epsilon>0$ there exists $\delta = \delta(d, \epsilon)>0$
such that $\varrho_{\bold{\Phi}}^{A}(\bar x)\leq d$ and $\varrho_{\bold{\Phi}}^{A}(\bar y)\leq \delta$ for $\bar{ x}, \bar{y} \in l_{\bold{\Phi}}^{A}(X)$.
Since $\Phi\in \delta_2$ and satisfies condition $(*)$, so by Lemma \ref{5}, there exists $d_1\in(0,1)$ such that $\|\bar x\|_{\bold{\Phi}}^{A}\leq d_1$. Also by using Lemma \ref{3}, for $\delta>0$ we have $\delta_1>0$ such that $\|\bar y\|_{\bold{\Phi}}^A \leq \delta_1$ whenever $\varrho_{\bold{\Phi}}^{A}( \bar{y})\leq \delta$ and $\bar y\in l_{\bold{\Phi}}^{A}(X)$. Since $\varrho_{\bold{\Phi}}^{A}$ is a convex and bounded on $l_{\bold{\Phi}}^{A}(X)$, by Lemma \ref{10}, we have $|\varrho_{\bold{\Phi}}^{A}(\bar x+ \bar y)-\varrho_{\bold{\Phi}}^{A}(\bar x)|<\epsilon$.
\end{proof}

\begin{Lemma}{\label{8}}
Let $\bold{\Phi}\in \delta_2$. Then for any $\epsilon>0$ there exists $\delta =\delta(\epsilon)>0$ such that $\varrho_{\bold{\Phi}}^{A}(\bar x)> \delta$ whenever
$\|\bar x\|_{\bold{\Phi}}^{A} \geq \epsilon$.
\end{Lemma}
\begin{proof}
The proof follows from Lemma \ref{3}.
\end{proof}

\begin{Lemma}{\label{7}}
Let $\bold{\Phi}\in \delta_2$ and satisfies the condition
$(*)$. Then for any $\bar{x}\in
l_{\bold{\Phi}}^{A}(X)$ and any $\epsilon>0$ there exists
$\delta=\delta(\epsilon)>0$ such that $\varrho_{\bold{\Phi}}^{A}(\bar{x})\geq
1+\epsilon$ implies $\|\bar{x}\|_{\bold{\Phi}}^A \geq 1+\delta$.
\end{Lemma}
Now we will prove the main result in this section.
\begin{Theorem}
Let $A=(a_{nk})$ be a triangle matrix. If a Musielak-Orlicz function $\bold{\Phi} = (\phi_n) \in \delta_2$ and satisfies condition $(*)$ given by $(\ref{2.1})$, then $ \displaystyle l^{A}_{\bold{\Phi}}(X)$ has the uniform Opial property.
\end{Theorem}

\begin{proof}
Let $ S(\displaystyle l^{A}_{\bold{\Phi}}(X))$ be the unit sphere of $\displaystyle l^{A}_{\bold{\Phi}}(X)$ and  $(\bar{x}^{(l)}) \subset S(\displaystyle l^{A}_{\bold{\Phi}}(X))$ be any weakly null sequence.
We will show that for any $\epsilon>0$ there exists $\mu>0$ such that
\begin{center}
$\displaystyle\liminf_{l\rightarrow\infty}\|\bar{x}^{(l)} + \bar{x} \|_{\bold{\Phi}}^A\geq
1+\mu$,
\end{center}
where $\bar{x} \in \displaystyle l^{A}_{\bold{\Phi}}(X)$ satisfying $\|\bar{x} \|_{\bold{\Phi}}^A \geq \epsilon$.
Since $\Phi \in \delta_2$ and $\|\bar{x} \|_{\bold{\Phi}}^A \geq \epsilon$, so by using Lemma \ref{8}, for each $\epsilon>0$ there exists $\delta\in (0, 1)$ such that
$\varrho_{\bold{\Phi}}^{A}(\bar x)\geq \delta$, and also by Lemma \ref{6} there
exists $\delta_1\in (0, \delta)$ such that $\varrho_{\bold{\Phi}}^{A}(\bar u)\leq 1$,
$\varrho_{\bold{\Phi}}^{A}(\bar v)\leq \delta_1$ implies
\begin{equation}{\label{eqn9}}
|\varrho_{\bold{\Phi}}^{A}(\bar u+ \bar v)-\varrho_{\bold{\Phi}}^{A}(\bar u)|<\frac{\delta}{6} \mbox{~for any ~}
\bar{u}, \bar{v} \in \displaystyle l^{A}_{\bold{\Phi}}(X).
\end{equation}
Since $\varrho_{\bold{\Phi}}^{A}(\bar x)<\infty$, so there exists $n_0\in
\mathbb{N}$ such that
\begin{equation}{\label{eqn10}}
\displaystyle\sum_{n=n_0 +
1}^{\infty}\phi_n \bigg(\sum\limits_{k=1}^{n}\|a_{nk}x_{k} \|\bigg)\leq\frac{\delta_1}{6}.
\end{equation}
From equation $(\ref{eqn10})$, it follows that
\begin{align*}
\delta \leq\displaystyle\sum_{n=1}^{n_0}\phi_n \bigg(\sum\limits_{k=1}^{n}\|a_{nk}x_{k} \|\bigg) + \displaystyle\sum_{n=n_0+1}^{\infty}\phi_n \bigg(\sum\limits_{k=1}^{n}\|a_{nk}x_{k} \|\bigg)
 \leq \displaystyle\sum_{n=1}^{n_0}\phi_n \bigg(\sum\limits_{k=1}^{n}\|a_{nk}x_{k} \|\bigg) + \frac{\delta_1}{6},
 \end{align*}
which implies $\displaystyle\sum_{n=1}^{n_0}\phi_n \bigg(\sum\limits_{k=1}^{n}\|a_{nk}x_{k} \|\bigg) \geq\delta-\frac{\delta_1}{6}>
\delta-\frac{\delta}{6} =\frac{5\delta}{6}.$
Since $\bar{x}^{(l)}\rightarrow 0$ weakly, so $\bar{x}^{(l)}_k \rightarrow 0$ for each
$k$. Hence there exists $l_0$ such that for all $l\geq l_0$ the last
inequality yields
\begin{equation}{\label{eqn11}}
\displaystyle\sum_{n=
1}^{n_0}\phi_n \bigg(\sum\limits_{k=1}^{n}\|a_{nk}({x}^{(l)}_k + x_{k})\|\bigg) \geq\frac{5\delta}{6}.
\end{equation}
Again $\bar{x}^{(l)} \rightarrow 0$ weakly, so we can choose $n_0 \in \mathbb{N}$ such that
$\varrho_{\bold{\Phi}}^{A}(\bar{x}^{(l)}|_{n_0})\rightarrow 0$ as $l\rightarrow\infty$. Therefore there exists $l_1 > l_0$ such that
$\varrho_{\bold{\Phi}}^{A}(\bar{x}^{(l)}|_{n_0})\leq\delta_1$ for all $l\geq l_1$. Since
$(\bar{x}^{(l)})\subset S(\displaystyle l^{A}_{\bold{\Phi}}(X))$, so by Lemma
\ref{4}, we have $\varrho_{\bold{\Phi}}^{A}(\bar{x}^{(l)}) =1$, and hence for $n_0 \in \mathbb{N}$, we have $\varrho_{\bold{\Phi}}^{A}(\bar{x}^{(l)}|_{\mathbb{N}-n_0})\leq 1$.
Let us set $\bar u=\bar{x}^{(l)}|_{\mathbb{N}-n_0}$ and $\bar v=\bar{x}^{(l)}|_{n_0}$. Then $\bar u,
\bar v\in \displaystyle l^{A}_{\bold{\Phi}}(X)$ and $\varrho_{\bold{\Phi}}^{A} (\bar u)\leq1$,  $\varrho_{\bold{\Phi}}^{A}(\bar v) \leq
\delta_1$. Using $(\ref{eqn9})$, we have
\begin{center}
$\big|\varrho_{\bold{\Phi}}^{A}(\bar{x}^{(l)}|_{\mathbb{N}-n_0}+
\bar{x}^{(l)}|_{n_0}\big)-\varrho_{\bold{\Phi}}^{A}\big(
\bar{x}^{(l)}|_{\mathbb{N}-n_0}\big)\big|<\frac{\delta}{6}$,
\end{center}
for all $l\geq l_1$. Thus $\varrho_{\bold{\Phi}}^{A}(\bar{x}^{(l)})-\frac{\delta}{6}<\varrho_{\bold{\Phi}}^{A}\big( \bar{x}^{(l)}|_{\mathbb{N}-n_0}\big)$ for all $l\geq l_1$, i.e., \\ $\displaystyle\sum_{n=n_0+1}^{\infty}\phi_n \Big(\sum\limits_{k=1}^{n}\|a_{nk}{x}^{(l)}_k\|\Big)$ $> 1-\frac{\delta}{6}$ for all
$l\geq l_1$. Again, since
$\varrho_{\bold{\Phi}}^{A}\big(\bar{x}^{(l)}|_{\mathbb{N}-n_0}\big)\leq 1$ and
$\varrho_{\bold{\Phi}}^{A}\big( \bar{x} |_{\mathbb{N}-n_0}\big)\leq
\frac{\delta_1}{6}<\delta_1$, so from the equations $(\ref{eqn9})$, we have
$\Big| \varrho_{\bold{\Phi}}^{A}\big(\bar{x}^{(l)} + \bar{x}|_{\mathbb{N}-n_0}\big) - \varrho_{\bold{\Phi}}^{A}\big(\bar{x}^{(l)}|_{\mathbb{N}-n_0}\big) \Big|
< \frac{\delta}{6}$.

We have for $l \geq l_1$
\begin{align*}
\varrho_{\bold{\Phi}}^{A}(\bar{x}^{(l)}+ \bar{x}) &= \displaystyle\sum_{n=1}^{n_0}\phi_n \bigg(\sum\limits_{k=1}^{n}\|a_{nk}({x}^{(l)}_k + x_{k} )\|\bigg) + \displaystyle\sum_{n=n_0+1}^{\infty}\phi_n \bigg(\sum\limits_{k=1}^{n}\|a_{nk}( {x}^{(l)}_k + x_{k} )\|\bigg) \\
 &>\displaystyle\sum_{n=1}^{n_0}\phi_n \bigg(\sum\limits_{k=1}^{n}\|a_{nk}({x}^{(l)}_k + x_{k} ) \|\bigg) + \displaystyle\sum_{n=n_0+1}^{\infty}\phi_n \bigg(\sum\limits_{k=1}^{n}\|a_{nk}{x}^{(l)}_k \|\bigg) -\frac{\delta}{6}\\
 & > \frac{5\delta}{6}+ \Big(1-\frac{\delta}{6}\Big)-\frac{\delta}{6}=1+ \frac{\delta}{2}.
 \end{align*}
Since $\bold{\Phi}\in\delta_2$ and satisfies the condition $(*)$, so by Lemma \ref{7} there exists $\mu>0$ depending only on
$\delta$ such that $\|\bar{x}^{(l)} + \bar{x}\|_{\bold{\Phi}}^A > 1+ \mu$ for $l \geq l_1$. Hence
$\displaystyle\liminf_{l\rightarrow\infty}\|\bar{x}^{(l)} +\bar{x}\|_{\bold{\Phi}}^A\geq
1+\mu$. This completes the proof.
\end{proof}

\begin{note}
 Let $A$ be an identity matrix, $X = \mathbb{R}$ and $\phi_n(u) = u^{p_n}$ for all $u \geq 0$, $1 \leq p_n< \infty$ such that $\displaystyle \sup_{n \geq 1}p_n < \infty$. Then $\bold{\Phi} = (\phi_n) \in \delta_2$ and satisfies the condition $(*)$ given by \eqref{2.1}. The space $ \displaystyle l_{\bold{\Phi}}$ has the uniform Opial property studied by Cui and Hudzik \cite{CUI2}. It would be interesting to find the necessary condition for the space $ \displaystyle l^{A}_{\bold{\Phi}}(X)$ to have the
 uniform Opial property.
\end{note}


\begin{Theorem}{\label{thm2}}
Let $X$ be a $\sigma$-DC Banach lattice. Then the space $ \displaystyle l^{A}_{\bold{\Phi}}(X)$ is a $\sigma$-DC Banach lattice.
\end{Theorem}
\begin{proof}
We have proved that $ \displaystyle l^{A}_{\bold{\Phi}}(X)$ is a Banach space with the norm $\|.\|_{\bold{\Phi}}^{A}.$
To prove $ \displaystyle l^{A}_{\bold{\Phi}}(X)$ is a Banach lattice, let $\bar{x}=(x_k), \bar{y}=(y_k) \in \displaystyle l^{A}_{\bold{\Phi}}(X)$
such that $|x_k| \leq |y_k |$ for all $k$. Since by hypothesis $X$ is a Banach lattice and $x_k, y_k \in X$ for all $k$. Therefore
$|x_k| \leq |y_k|$ $\Rightarrow$ $\|x_k \| \leq \| y_k \|$ for all $k$.
Let $\sigma > 0$ be any number. Now using the nondecreasing property of each $\phi_n$, we get
\begin{align*}
\sum_{n=1}^{\infty}\phi_n \bigg(\frac{\sum_{k=1}^{\infty}|a_{nk}|\|x_k \|}{\sigma} \bigg)  \leq \sum_{n=1}^{\infty}\phi_n \bigg(\frac{\sum_{k=1}^{\infty}|a_{nk}|\|y_k \|}{\sigma} \bigg)
\quad \Rightarrow \|\bar{x}\|_{\bold{\Phi}}^{A} & \leq \|\bar{y}\|_{\bold{\Phi}}^{A}.
\end{align*}
Hence $ \displaystyle l^{A}_{\bold{\Phi}}(X)$ is a Banach lattice.\\
Now we shall show that the space $ \displaystyle l^{A}_{\bold{\Phi}}(X)$ is $\sigma$-DC. Let $(\bar{x}^{(n)})$ be a non-negative
order bounded sequence and bounded above by $\bar{y} \in \displaystyle l^{A}_{\bold{\Phi}}(X)$, i.e., for each $k \in \mathbb{N}$,
$ x_{k}^{(n)} \leq y_k \quad \mbox{for all}~ n \geq 1.$

Since $X$ is a $\sigma$-DC, there exists $(x_k) \subset X$ such that $\displaystyle\sup_{n}x_k^{(n)} = x_k$ for all $k$. Therefore $x_k \leq y_k$ for all $k$.
Again $\bar{y} \in l_{\bold{\Phi}}^{A}(X)$, so
$\displaystyle\sum_{n=1}^{\infty}\phi_n \bigg(\frac{\displaystyle\sum_{k=1}^{\infty}|a_{nk}|\|x_k \|}{\sigma} \bigg) < \infty$ for some $\sigma > 0$, and hence
$ \bar x \in \displaystyle l^{A}_{\bold{\Phi}}(X)$. This completes the proof.
\end{proof}

\begin{Corollary}
If $X$ is a Banach lattice, then $\displaystyle h^{A}_{\bold{\Phi}}(X)$ is a Banach lattice.
\end{Corollary}

\begin{Theorem}
Let a Banach lattice $X$ be an AL-space. If $\bold{\Phi} =(\phi_n) \in \delta_2$ and satisfies the condition $(*)$ given by $\eqref{2.1}$, then the Banach lattice $\displaystyle l^{A}_{\bold{\Phi}}(X)$ is uniformly monotone.
\end{Theorem}

\begin{proof}
Let $\epsilon >0$ be any number. Let $\bar{x}=(x_k), \bar{y}=(y_k) \in \displaystyle l^{A}_{\bold{\Phi}}(X)$ such that $\bar 0 \leq \bar{x} \leq \bar{y}$ with $\|\bar{x}\|^{A}_{\bold{\Phi}}=1$ and $ \|\bar{y}\|^{A}_{\bold{\Phi}} \geq \epsilon$. Since $\bold{\Phi} =(\phi_n) \in \delta_2$, from Lemma \ref{4}, we get $\varrho_{\bold{\Phi}}^{A}(\bar{x}) =1$ as $\|\bar{x}\|^{A}_{\bold{\Phi}}=1$ and from Lemma \ref{8}, there exists $\delta(\epsilon)> 0$ such that $\varrho_{\bold{\Phi}}^{A}(\bar{y}) > \delta(\epsilon)$ as $ \|\bar{y}\|^{A}_{\bold{\Phi}} \geq \epsilon$.

Now if $u, v \geq 0$ and $\phi$ is an Orlicz function, then we will prove that $\phi(u+ v) \geq \phi(u) + \phi(v)$. With out loss of generality, we
assume that $u > v>0$. Clearly $u < u +v$ and $v < u +v$. Then $\phi(u) = \phi(v . \frac{u}{v}) \geq \frac {u}{v}\phi(v)$ as $\phi $ is an Orlicz function.
Thus for $u > v$, we have $\frac{\phi(u)}{u} \geq \frac {\phi(v)}{v}$. Since  $u +v >u$ and $u +v >v$, we get
$$\frac{\phi(u+v)}{u+v} \geq  \frac {\phi(u)}{u} \quad \mbox{and} \quad \frac{\phi(u+v)}{u+v} \geq \frac {\phi(v)}{v}.$$
Therefore $\phi(u+v) \geq {\phi(u)} + {\phi(u)}$ for $u, v \geq 0$.

Now $\varrho_{\bold{\Phi}}^{A}(\bar{x} + \bar{y}) = \displaystyle \sum_{n=1}^{\infty}\phi_n \Big( \sum_{k=1}^{\infty}\|a_{nk}(x_k + y_k) \|\Big)$.
Since $X$ is an AL-space, so $ \displaystyle \sum_{k=1}^{\infty}\|a_{nk}(x_k + y_k) \| = \sum_{k=1}^{\infty} \|a_{nk}x_k \| +  \sum_{k=1}^{\infty} \|a_{nk}y_k \|$.
Take $ \displaystyle\sum_{k=1}^{\infty} \|a_{nk}x_k \|=u $ and $ \displaystyle\sum_{k=1}^{\infty} \|a_{nk}y_k \| = v$. Therefore for each Orlicz function $\phi_n$, we have
$$\phi_{n}\Big ( \sum_{k=1}^{\infty} \|a_{nk}x_k \| +  \sum_{k=1}^{\infty} \|a_{nk}y_k \| \Big) \geq \phi_{n}\Big ( \sum_{k=1}^{\infty} \|a_{nk}x_k \| \Big) + \phi_{n}\Big ( \sum_{k=1}^{\infty} \|a_{nk}y_k \| \Big).$$
Hence
$$\varrho_{\bold{\Phi}}^{A}(\bar{x} + \bar{y}) \geq \varrho_{\bold{\Phi}}^{A}(\bar{x}) + \varrho_{\bold{\Phi}}^{A}(\bar{y}).$$
Therefore $\varrho_{\bold{\Phi}}^{A}(\bar{x} + \bar{y}) \geq 1+ \delta(\epsilon).$
Since $\bold{\Phi} =(\phi_n) \in \delta_2$ and satisfies the condition $(*)$, by using Lemma \ref{7}, there exists $\mu > 0$ independent of $\bar{x}, \bar{y}$ such that $\|\bar{x} + \bar{y} \|^{A}_{\bold{\Phi}} \geq 1 + \mu.$ This finishes the proof.
\end{proof}

\begin{Corollary}
If a Banach lattice $X$ is an AL-space and $\bold{\Phi} =(\phi_n)$ satisfies the condition $(*)$, then $\displaystyle h^{A}_{\bold{\Phi}}(X)$ is strictly monotone.
\end{Corollary}

\begin{Theorem}
Let $A$ be a triangle. Then $\displaystyle h^{A}_{\bold{\Phi}}(X)$ is the subspace of all order continuous elements of $\displaystyle l^{A}_{\bold{\Phi}}(X)$.
\end{Theorem}

\begin{proof}
Let $\epsilon> 0$ be any number and $\bar{x} =(x_n) \in \displaystyle h^{A}_{\bold{\Phi}}(X)$. We will show that $\bar{x}$ is order continuous.
Since $\bar{x} \in \displaystyle h^{A}_{\bold{\Phi}}(X)$, so there is $t> 0$ and $n_0 \in \mathbb{N}$ such that
$$\displaystyle\sum_{n=n_0 + 1}^{\infty} \phi_n \bigg(t \sum\limits_{k =1 }^{n}\|a_{nk}x_{k} \|\bigg) < \frac{\epsilon}{2}.$$
Suppose $(\bar{x}^{(m)})$ be a sequence in $\displaystyle h^{A}_{\bold{\Phi}}(X)$ such that $\bar{x}^{(m)} \rightarrow \bar 0$ coordinate wise and $\bar 0 \leq \bar{x}^{(m)} \leq |\bar{x}|$ for all $m \in \mathbb{N}$.
Let us denote for all $n \in \mathbb{N}$,
\begin{align*}
\phi_n \bigg(t \sum\limits_{k =1 }^{n}\|a_{nk}x_{k} \|\bigg)  = \beta(n) \mbox{~~and~~}
\phi_n \bigg(t \sum\limits_{k =1 }^{n}\|a_{nk}x_{k}^{(m)} \|\bigg) = \beta^{(m)}(n).
\end{align*}
Since $x_{k}^{(m)} \rightarrow \bar 0$ and each $\phi_n$ is continuous, we get
$\beta^{(m)}(n) \rightarrow 0$ as $m \rightarrow \infty$ for all $n \in \mathbb{N}$.
Therefore we can choose $n_1 \in \mathbb{N}$ such that
$\displaystyle\sum_{n=1}^{n_0}\beta^{(m)}(n) < \frac{\epsilon}{2}$
for all $m \geq n_1$. Since $\bar{x}^{(m)} \leq |\bar{x}|$ for all $m \in \mathbb{N}$, we have
\begin{center}
$\displaystyle\sum_{n=n_0 +1}^{\infty}\beta^{(m)}(n) \leq \sum_{n=n_0 +1 }^{\infty}\beta(n) < \frac{\epsilon}{2}.$
\end{center}
Thus for all $m \geq n_1$ and $t> 0$,
\begin{align*}
\varrho_{\bold{\Phi}}^{A}(t\bar{x}^{(m)}) & =  \sum_{n=1}^{n_0} \phi_n \bigg(t \sum\limits_{k =1 }^{n}\|a_{nk}x_{k}^{(m)} \|\bigg) +
\sum_{n=n_0 + 1}^{\infty}\phi_n \bigg(t \sum\limits_{k =1 }^{n}\|a_{nk}x_{k}^{(m)} \|\bigg)\\
& < \frac{\epsilon}{2} + \frac{\epsilon}{2} = \epsilon.
\end{align*}
Therefore any arbitrary $t>0$, we get $\varrho_{\bold{\Phi}}^{A}(t\bar{x}^{(m)}) \rightarrow 0$ as $m \rightarrow \infty$, and hence
$\|x^{(m)}\| \rightarrow 0$. Thus $\bar{x}$ is an order continuous in $\displaystyle h^{A}_{\bold{\Phi}}(X)$. Arbitrariness of $\bar{x}$
implies that the space $\displaystyle h^{A}_{\bold{\Phi}}(X)$ is order continuous.
\end{proof}

\subsection{Operators of $s$-type $l_{\bold{\Phi}}^{A}$}

Let $E, F$ be two Banach spaces. Then
$$\mathcal{L}_{{\bold{\Phi}}}^{A}(E, F) = \Big \{ T \in \mathcal{L}(E, F): (s_n(T)) \in l_{\bold{\Phi}}^{A} \Big \}.$$
For $T \in \mathcal{L}_{\bold{\Phi}}^{A}(E, F)$, we define

$$ \| T \|_{\bold{\Phi}}^{A} = \inf \Big \{ \sigma > 0 : \sum_{n=1}^{\infty} \phi_n \bigg(\frac{\sum\limits_{k=1}^{\infty}|a_{nk}s_{k}(T) |}{ \sigma}\bigg) \leq 1 \Big \}.$$

We denote $\mathcal{L}_{{\bold{\Phi}}}^{A}$ as the class of $s$-type $l_{\bold{\Phi}}^{A}$ operators between any two arbitrary Banach spaces. We also define
$$\mathcal{H}_{{\bold{\Phi}}}^{A}(E, F) = \Big \{ T \in \mathcal{L}(E, F): (s_n(T)) \in h_{\bold{\Phi}}^{A} \Big \}.$$
Let $A = (a_{nk})$ be a matrix from the collection $\mathcal{A}$ satisfying the condition
\begin{align}\label{equa1}
		|a_{n,2k-1}| + |a_{n,2k}| \leq M |a_{nk}| \qquad {\rm{for ~ each ~}} k~ {\rm{and}}~ n,
\end{align}
		where $M $ is a constant independent of $n$ and $k$.

\begin{note}
It is easy to give example of matrices which satisfy condition $(\ref{equa1})$. For example,\\
1. { N$\ddot{o}$rlund matrix
        $ A = (a_{nk})$, where $a_{nk}$ is defined as
\begin{displaymath}
    a_{nk}  = \left\{
     \begin{array}{ll}
        { a_{n+1-k}\over A_{n}} & : 1 \leq k \leq n \\
        0          & :  k > n
     \end{array}
   \right.
\end{displaymath}
       where $ a_{n} $ is non negative for each $ n $ and $ A_{n} = \sum \limits_{k =1}^{n} a_{k} >0 $.}\\
2. Hilbert matrix $A=(a_{nk})$, where
\begin{displaymath}
    a_{nk}  = \left\{
     \begin{array}{ll}
         \frac{1}{n+k-1} & : 1 \leq n, k < \infty. \\
     \end{array}
   \right.
\end{displaymath}
\end{note}

\begin{pro}
Let  $A = (a_{nk})$ be an infinite matrix such that $A = (a_{nk}) \in \mathcal{A}$ and satisfying the condition $(\ref{equa1})$. Then $(\mathcal{L}_{\bold{\Phi}}^{A}(E, F), \| . \|_{\bold{\Phi}}^{A})$ is a quasi-Banach space. Moreover, the inclusion map from $(\mathcal{L}_{\bold{\Phi}}^{A}(E, F), \| . \|_{\bold{\Phi}}^{A})$ to $(\mathcal{L}(E, F), \| . \|)$ is continuous.
\end{pro}
\begin{proof}
Let $S, T \in \mathcal{L}_{\bold{\Phi}}^{A}(E, F)$. Then there exist $\sigma_1> 0, \sigma_2> 0$ such that
\begin{align*}
\sum_{n=1}^{\infty} \phi_n \bigg(\frac{\sum\limits_{k=1}^{\infty}|a_{nk}s_{k}(S) |}{ \sigma_1}\bigg) < \infty \mbox{~~and~~}
\sum_{n=1}^{\infty} \phi_n \bigg(\frac{\sum\limits_{k=1}^{\infty}|a_{nk}s_{k}(T) |}{ \sigma_2}\bigg) < \infty.
\end{align*}
Now using the non-increasing property of $s$-number and the condition $(\ref{equa1})$ on the matrix $A=(a_{nk})$, we get
\begin{align}\label{equa2}
\sum\limits_{k=1}^{\infty}|a_{nk}s_{k}(S + T) |
& \leq  M \sum\limits_{k=1}^{\infty}|a_{n,k}|\big( s_{k}(S) + s_{k}(T) \big)
\end{align}
Using $(\ref {equa2})$ and convexity property of each $\phi_n$, we have

\begin{align*}
\sum_{n=1}^{\infty} \phi_n \bigg(\frac{\sum\limits_{k=1}^{\infty}|a_{nk}s_{k}(S + T) |}{ M(\sigma_1 + \sigma_2) }\bigg)
& \leq  \frac{\sigma_1}{\sigma_1 + \sigma_2}\bigg[\sum_{n=1}^{\infty} \phi_n \bigg(\frac{\sum\limits_{k=1}^{\infty}|a_{nk}s_{k}(S)|}{ \sigma_1 }\bigg) \bigg]
+ \frac{\sigma_2}{\sigma_1 + \sigma_2} \bigg[\sum_{n=1}^{\infty} \phi_n \bigg(\frac {\sum\limits_{k=1}^{\infty}|a_{nk}s_{k}(T)|}{ \sigma_2 }\bigg) \bigg]\\
& < \infty.
\end{align*}
This shows that $S + T \in \mathcal{L}_{\bold{\Phi}}^{A}(E, F).$

Let $ \alpha \in \mathbb{R}$ and $T \in \mathcal{L}_{\bold{\Phi}}^{A}(E, F).$ If $\alpha =0$, then it is trivial. Suppose $\alpha \neq 0$. Then
\begin{align*}
\sum_{n=1}^{\infty} \phi_n \bigg(\frac{\sum\limits_{k=1}^{\infty}|a_{nk}s_{k}(\alpha T) |}{ \sigma}\bigg) & \leq \sum_{n=1}^{\infty} \phi_n \bigg( |\alpha|\frac{\sum\limits_{k=1}^{\infty}|a_{nk}s_{k}(T) |}{ \sigma}\bigg)
= \sum_{n=1}^{\infty} \phi_n \bigg(\frac{\sum\limits_{k=1}^{\infty}|a_{nk}s_{k}(T) |}{\frac{\sigma}{|\alpha|}}\bigg)< \infty
\end{align*}
as $T \in \mathcal{L}_{\bold{\Phi}}^{A}(E, F)$.
This shows that $\alpha T \in  \mathcal{L}_{\bold{\Phi}}^{A}(E, F)$. Hence $ \mathcal{L}_{\bold{\Phi}}^{A}(E, F)$ is a linear space.

 To show $\| . \|_{\bold{\Phi}}^{A}$ is a quasi-norm on the space $\mathcal{L}_{\bold{\Phi}}^{A}(E, F)$.
Let $T \in \mathcal{L}_{\bold{\Phi}}^{A}(E, F) $ such that $\| T \|_{\bold{\Phi}}^{A}= 0$. Then for all $\epsilon > 0$, we have
$$\sum_{n=1}^{\infty} \phi_n \bigg(\frac{\sum\limits_{k=1}^{\infty}|a_{nk}s_{k}(T) |}{ \epsilon}\bigg) \leq 1.$$
Therefore the sequence $\bigg(\frac{\sum\limits_{k=1}^{\infty}|a_{nk}s_{k}(T) |}{ \epsilon}\bigg)$ is bounded, so there exists $C > 0$ such that
$\frac{\sum\limits_{k=1}^{\infty}|a_{nk}s_{k}(T) |}{ \epsilon}  \leq C \mbox{~for all~}n.$ Since $(a_{nk}) \in \mathcal{A}$, there exists $n_0 \in \mathbb{N}$
such that $a_{n_{0}1} \neq 0$ and hence
\begin{align}{\label{equa3.1}}
|a_{n_{0}1}s_{1}(T) | & \leq C \epsilon \mbox{~for all~}n
\end{align}
which is true for any arbitrary $\epsilon>0$. Thus $\|T \|= s_1(T) = 0$, and hence $T =0$.

Let $S, T \in \mathcal{L}_{\bold{\Phi}}^{A}(E, F) $ and $\epsilon > 0$ be any positive number. Choose $\sigma_1 > 0, \sigma_2 > 0$ such that
\begin{center}
$\displaystyle\sum_{n=1}^{\infty} \phi_n \bigg(\frac{\sum\limits_{k=1}^{\infty}|a_{nk}s_{k}(S) |}{ \sigma_1}\bigg) \leq 1, ~~~~ \sigma_1 \leq \| S \|_{\bold{\Phi}}^{A} + \frac{\epsilon}{2}$ and
$\displaystyle\sum_{n=1}^{\infty} \phi_n \bigg(\frac{\sum\limits_{k=1}^{\infty}|a_{nk}s_{k}(T) |}{ \sigma_2}\bigg) \leq 1
, ~~~~\sigma_2 \leq \| T \|_{\bold{\Phi}}^{A} + \frac{\epsilon}{2} \quad \mbox{hold}.$
\end{center}
With out loss of generality, we can choose $M > 1$. Now from the above, we have
\begin{align*}
\sum_{n=1}^{\infty} \phi_n \bigg(\frac{\sum\limits_{k=1}^{\infty}|a_{nk}s_{k}(S + T) |}{ M(\sigma_1 + \sigma_2) }\bigg) \leq 1
\end{align*}
which implies
\begin{center}
$\| S+ T \|_{\bold{\Phi}}^{A} \leq M[\sigma_1 + \sigma_2] \leq M[\| S\|_{\bold{\Phi}}^{A} + \| T \|_{\bold{\Phi}}^{A} + \epsilon].$
\end{center}
Since $\epsilon > 0 $ is arbitrary, so we have
\begin{center}
$\| S+ T \|_{\bold{\Phi}}^{A} \leq M[\| S\|_{\bold{\Phi}}^{A} + \| T \|_{\bold{\Phi}}^{A}].$
\end{center}
Hence $\| .\|_{\bold{\Phi}}^{A}$ is a quasi-norm on the space $\mathcal{L}_{\bold{\Phi}}^{A}(E, F)$.

To prove completeness, let $(T^{(m)})$ be a Cauchy sequence in $\mathcal{L}_{\bold{\Phi}}^{A}(E, F)$.
Then for each $\epsilon > 0$, there exists $m_0 \in \mathbb{N}$ such that
\begin{align}\label{equa3}
\sum_{n=1}^{\infty} \phi_n \bigg(\frac{\sum\limits_{k=1}^{\infty}|a_{nk}s_{k}(T^{(l)} - T^{(m)}) |}{ \epsilon}\bigg) \leq 1 ~ \mbox{for all~} m, l \geq m_0.
\end{align}
Therefore the sequence $ \bigg(\frac{\sum\limits_{k=1}^{\infty}|a_{nk}s_{k}(T^{(l)} - T^{(m)}) |}{ \epsilon}\bigg)$ is bounded.
Using the same argument as above, we have $\|T^{(l)} - T^{(m)} \| \rightarrow 0$ as $l, m \rightarrow \infty$.
Thus $(T^{(m)})$ is a Cauchy sequence in $\mathcal{L}(E, F)$, and hence converges. Let $T = \displaystyle\lim_{m \rightarrow \infty}T^{(m)}$.
Also $s_{k}(T^{(l)} - T^{(m)}) \rightarrow s_{k}(T^{(l)} - T)$ as $m \rightarrow \infty$ for each $k \in \mathbb{N}$.
Since each $\phi_n$ is continuous so taking $m \rightarrow \infty$, we get from $(\ref{equa3})$,
\begin{align*}
\sum_{n=1}^{\infty} \phi_n \bigg(\frac{\sum\limits_{k=1}^{\infty}|a_{nk}s_{k}(T^{(l)} - T) |}{ \epsilon}\bigg) \leq 1 ~~ \mbox{for all}~~ l~ \geq m_0.
\end{align*}
Thus $(T^{(l)})$ converges to $T$ in $\mathcal{L}_{\bold{\Phi}}^{A}(E, F)$.
In particular $T - T^{(m_0)} \in \mathcal{L}_{\bold{\Phi}}^{A}(E, F)$, so $T = T^{(m_0)} + (T -T^{(m_0)}) \in \mathcal{L}_{\bold{\Phi}}^{A}(E, F)$.\\

To prove the second part, we have from (\ref{equa3.1})
\begin{align*}
\| T \| & \leq \frac{C}{|a_{n_{0}1}|} \inf \bigg\{ \epsilon> 0: \sum_{n=1}^{\infty} \phi_n \bigg(\frac{\sum\limits_{k=1}^{\infty}|a_{nk}s_{k}(T) |}{ \epsilon}\bigg) \leq 1 \bigg \}\\
\mbox{i.e.,}~\| T \| & \leq \frac{C}{|a_{n_{0}1}|} \| T \|_{\bold{\Phi}}^{A}.
\end{align*}
Hence the inclusion map from $(\mathcal{L}_{\bold{\Phi}}^{A}(E, F), \| . \|_{\bold{\Phi}}^{A})$ to $(\mathcal{L}(E, F), \| . \|)$ is continuous.
This completes the proof.
\end{proof}

\begin{Remark}
1. If we take the matrix $A={(a_{nk})}$ as an identity matrix and $s$-number as approximation number then the $s$-type $l_{\bold{\Phi}}^{A}$ operators become $l_p$ type $\cite{PIE1}$ and $l_{\phi }$ type $\cite{GUP1}$ operators when $\phi_n= x^p $ for $0< p< \infty$ and $\phi_n = \phi$, an Orlicz function respectively.\\
2. If we take the matrix $A={(a_{nk})}$ such that the matrix satisfies the condition $(\ref{equa1})$, then the class of $s$-type $l_{\bold{\Phi}}^{A}$ operators becomes $s$-type $|A, p|$  operators introduced by Maji and Srivastava $\cite{AMIT}$.
\end{Remark}

\begin{Theorem}{\label{thm4}}
Let $A = (a_{nk})$ be an infinite matrix such that $A = (a_{nk}) \in \mathcal{A}$. If $A = (a_{nk})$ satisfies the condition $(\ref{equa1})$ and $(|a_{n1}|) \in l_{\bold{\Phi}}$, then the class $\mathcal{L}_{\bold{\Phi}}^{A}$ is an operator ideal.
\end{Theorem}

\begin{proof}
Let $E, F$ be any two Banach spaces and $\mathcal{L}_{\bold{\Phi}}^{A}(E, F)$ be any one of the component of $\mathcal{L}_{\bold{\Phi}}^{A}$. To prove $\mathcal{L}_{\bold{\Phi}}^{A}$ is an operator ideal, it is enough to prove $(OI1)$ and $(OI3)$ in Definition \ref{def2}. Let $x^{'} \in E^{'}$ and $y \in F$. Then $x^{'}\otimes y: E \rightarrow F$ is a rank one operator, and hence $s_{k}(x^{'}\otimes y) =0$ for $k \geq 2$. Then for some $\sigma > 0$
$$\sum_{n=1}^{\infty} \phi_n \bigg(\frac{\sum\limits_{k=1}^{\infty}|a_{nk}s_{k}(x^{'}\otimes y) |}{ \sigma}\bigg)= \|x^{'}\otimes y \| \sum_{n=1}^{\infty} \phi_n \bigg(\frac{\sum\limits_{k=1}^{\infty}|a_{n1}|}{ \sigma}\bigg)< \infty $$
as $(|a_{n1}|) \in l_{\bold{\Phi}}$. Thus $x^{'}\otimes y \in \mathcal{L}_{\bold{\Phi}}^{A}(E, F)$.

Let $T \in \mathcal{L}(E_{0}, E)$, $R \in \mathcal{L}(F, F_{0})$ and
        $ S \in \mathcal{L}_{\bold{\Phi}}^{A}{(E, F)}$. It is required to prove $ RST \in
        \mathcal{L}_{\bold{\Phi}}^{A}{(E_0, F_0)}$.\\
Using the property $(S3)$ in Definition \ref{def1}, we have $ s_{n}(RST) \leq \| R \| s_{n}(S) \| T \| $ for all $ n \in \mathbb{N} $.
Since $ S \in \mathcal{L}_{\bold{\Phi}}^{A}{(E, F)}$, there exists some $\sigma > 0$ such that
$\displaystyle\sum_{n=1}^{\infty} \phi_n \bigg(\frac{\sum\limits_{k=1}^{\infty}|a_{nk}s_{k}(S)|}{ \sigma}\bigg) < \infty.$
Therefore

\begin{align*}
		\sum_{n=1}^{\infty} \phi_n \bigg(\frac{\sum\limits_{k=1}^{\infty}|a_{nk}s_{k}(RST)|}{\| R \| \| T \| \sigma}\bigg) <\infty .
\end{align*}
       Thus $ RST \in \mathcal{L}_{\bold{\Phi}}^{A}{(E_{0}, F_{0})}$ and therefore $(OI3)$ is proved.
        Hence $ \mathcal{L}_{\bold{\Phi}}^{A}$ is an operator ideal.

To prove $\| .\|_{\bold{\Phi}}^{A}$ is an ideal-norm on $ \mathcal{L}_{\bold{\Phi}}^{A}$, let $ S \in \mathcal{L}_{\bold{\Phi}}^{A}{(E, F)}$. Then for given $\epsilon> 0$, there exists some $\sigma_{0}$ such that $\sigma_0 < \|S \|_{\bold{\Phi}}^{A} + \epsilon$ with
$$\sum_{n=1}^{\infty} \phi_n \bigg(\frac{\sum\limits_{k=1}^{\infty}|a_{nk}s_{k}(S)|}{ \sigma_{0}}\bigg) \leq 1.$$
Thus for $T \in \mathcal{L}(E_{0}, E)$, $R \in \mathcal{L}(F, F_{0})$, we have
\begin{align*}
		\sum_{n=1}^{\infty} \phi_n \bigg(\frac{\sum\limits_{k=1}^{\infty}|a_{nk}s_{k}(RST)|}{\| R \| \| T \| \sigma_{0}}\bigg) \leq 1.
\end{align*}
Hence $\| RST \|_{\bold{\Phi}}^{A} \leq \| R \| \| T \| \sigma_{0}< \| R \| \| T \|(\|S \|_{\bold{\Phi}}^{A} + \epsilon)$. As $\epsilon> 0$ is arbitrary,
we have $\| RST \|_{\bold{\Phi}}^{A} \leq \| R \| \|S \|_{\bold{\Phi}}^{A}\| T \|$. Thus $ \mathcal{L}_{\bold{\Phi}}^{A}$ is a quasi-Banach operator ideal.
\end{proof}

\begin{Remark}
In particular if we choose $\phi_n = \phi$, an Orlicz function for all $n$ and $A=(a_{nk})$ as a diagonal matrix $a_{nn}= n^{\frac{1}{p} -\frac{1}{q}}$ for $0<p, q< \infty$, the operator ideal $ \mathcal{L}_{\bold{\Phi}}^{A}$ becomes $ \mathcal{L}_{p, q, \Phi}$ studied by Gupta and Bhar
$\cite{GUP1}$. This example also shows that the condition $(\ref{equa1})$ on the matrix $A$ is only sufficient condition to form operator ideal.
\end{Remark}


%
%

\begin{pro}
The space $\mathcal{H}_{\bold{\Phi}}^{A}(E, F)$ is a closed subspace of $\mathcal{L}_{\bold{\Phi}}^{A}(E, F)$.
\end{pro}
\begin{proof}
Clearly $\mathcal{H}_{\bold{\Phi}}^{A}(E, F)$ is a subspace of $\mathcal{L}_{\bold{\Phi}}^{A}(E, F).$  To show $\mathcal{H}_{\bold{\Phi}}^{A}(E, F)$ is a closed subspace of $\mathcal{L}_{\bold{\Phi}}^{A}(E, F)$, let $T$ belongs to the closure of the space $\mathcal{H}_{\bold{\Phi}}^{A}(E, F)$ in the norm topology of $\mathcal{L}_{\bold{\Phi}}^{A}(E, F)$. Then there exists a sequence $(T^{(m)})$ in $\mathcal{H}_{\bold{\Phi}}^{A}(E, F)$ such that $\displaystyle\lim_{n \rightarrow \infty}\|T^{(m)} -T\|_{\bold{\Phi}}^{A} =0.$ Thus for $\epsilon> 0$, there exists $n_0 \in \mathbb{N}$ such that
$$\|T^{(m)} -T\|_{\bold{\Phi}}^{A} < \frac{\epsilon}{2} ~ \mbox{for all~} m \geq n_0.$$
Now using the condition $(\ref {equa1})$ on the matrix $A =(a_{nk})$ and $(\ref {equa2})$, we get
\begin{align*}
\sum_{n=1}^{\infty} \phi_n \bigg(\frac{\sum\limits_{k=1}^{\infty}|a_{nk}s_{k}(T) |}{ M \epsilon}\bigg)
& \leq \frac{1}{2} \Big[\sum_{n=1}^{\infty} \phi_{n} \bigg(\frac{\sum\limits_{k=1}^{\infty}|a_{n, k}s_{k}(T- T^{(n_{0})}) |}{\frac{\epsilon}{2}}\Big) + \sum_{n=1}^{\infty} \phi_{n} \bigg(\frac{\sum\limits_{k=1}^{\infty}|a_{n, k}s_{k}(T^{(n_{0})}) |}{\frac{\epsilon}{2}}\Big) \Big]
 < \infty,
\end{align*}
as $T^{(n_0)} \in \mathcal{H}_{\bold{\Phi}}^{A}(E, F)$ and $\displaystyle\sum_{n=1}^{\infty} \phi_n \bigg(\frac{\sum\limits_{k=1}^{\infty}|a_{nk}s_{k}(T -T^{(n_0)})|}{\|T -T^{(n_0)}\|_{\bold{\Phi}}^{A}} \Big)< \infty$. Thus $T \in \mathcal{H}_{\bold{\Phi}}^{A}(E, F)$ and hence the proof is complete.
\end{proof}

\begin{pro}
        If the $s$-number sequence is injective, then the quasi-Banach operator ideal $[\mathcal{L}_{\bold{\Phi}}^{A}, \| . \|_{\bold{\Phi}}^{A}]$ is injective.
\end{pro}
\begin{proof}
        Let $T \in \mathcal{L}(E, F)$ and $J \in \mathcal{L}(F, F_{0}) $ be any metric injection.
        Suppose that $JT \in  \mathcal{L}_{\bold{\Phi}}^{A}{(E, F_{0})}$. Then for some $\sigma_0 > 0$, we have
        $$\sum_{n=1}^{\infty} \phi_n \bigg(\frac{\sum\limits_{k=1}^{\infty}|a_{nk}s_{k}(JT) |}{ \sigma_0}\bigg) < \infty.$$
        Since the $s$-number sequence $s=(s_{n})$ is injective, we have $s_{n}(T)=s_{n}(JT)$ for all
        $T \in \mathcal{L}(E, F)$, $n\in \mathbb{N}$. Hence
        $$\sum_{n=1}^{\infty} \phi_n \bigg(\frac{\sum\limits_{k=1}^{\infty}|a_{nk}s_{k}(T) |}{ \sigma_0}\bigg) = \sum_{n=1}^{\infty} \phi_n \bigg(\frac{\sum\limits_{k=1}^{\infty}|a_{nk}s_{k}(JT) |}{ \sigma_0}\bigg) < \infty.$$
        Thus $ T \in \mathcal{L}_{\bold{\Phi}}^{A}{(E, F_{0})}$  and clearly
        $ \| JT \|_{\bold{\Phi}}^{A} = \| T \|_{\bold{\Phi}}^{A}$ holds. Hence the operator ideal $\mathcal{L}_{\bold{\Phi}}^{A}$ is injective.
        This completes the proof.
\end{proof}

\begin{pro}
        If the $s$-number sequence is surjective, then the quasi-Banach operator ideal $[\mathcal{L}_{\bold{\Phi}}^{A}, \| . \|_{\bold{\Phi}}^{A}]$ is surjective.
\end{pro}

\begin{proof}
We omit the proof as it follows in similar lines from the preceding proposition.
\end{proof}

\end{document}